\documentclass[12pt,a4paper]{article}
\usepackage{amsmath,amssymb,amsthm,mathtools}
\usepackage[utf8]{inputenc}
\usepackage[english]{babel}
\usepackage[T1]{fontenc}
\usepackage{lmodern}
\usepackage{microtype}
\usepackage{enumitem}
\usepackage{hyperref}
\usepackage{indentfirst}
\usepackage[nameinlink]{cleveref}
\usepackage[a4paper,margin=4cm]{geometry}
\hypersetup{colorlinks=true,linkcolor=black,citecolor=black,urlcolor=black}

\newtheorem{theorem}{Theorem}
\newtheorem{lemma}{Lemma}
\newtheorem{proposition}{Proposition}
\newtheorem{corollary}{Corollary}
\theoremstyle{remark}
\newtheorem{remark}{Remark}
\theoremstyle{definition}

\DeclareMathOperator{\cont}{cont}
\newcommand{\Disc}{\operatorname{Disc}}
\newcommand{\Jac}{\operatorname{Jac}}

\title{On the Irreducibility of the Cuboid Polynomial $P_{a,u}(t)$}

\author{Valery Asiryan\\[4pt]
{\small \texttt{asiryanvalery@gmail.com}}}
\date{\small October 9, 2025}

\begin{document}
\maketitle

% ===================== Abstract & Keywords =====================

\begin{abstract}
In this paper we consider the even monic degree-8 cuboid polynomial $P_{a,u}(t)$ with coprime integers $a\neq u>0$.
We prove irreducibility over $\mathbb{Z}$ by excluding all degree-8 splittings.
First, any putative $4{+}4$ factorization is shown to force a specific Diophantine constraint which has no integer solutions by a short $2$- and $3$-adic analysis.
Second, we exclude every $2{+}6$ factorization via an exact divisor criterion and a discriminant obstruction. 
Finally, after ruling out $2{+}6$, the patterns $2{+}2{+}4$, $2{+}2{+}2{+}2$, and $3{+}3{+}2$ regroup trivially to $2{+}6$ and are therefore impossible.
Consequently, $P_{a,u}(t)$ admits no nontrivial factorization in $\mathbb{Z}[t]$.
\end{abstract}

{\footnotesize

\textbf{Secondary}: 11D72 (Equations in many variables; Diophantine equations), 11S05 (Local and $p$-adic fields), 11Y05 (Factorization; primality).
}

% ===================== Statement =====================

\section{Problem Statement and Notation}
Let \(a,u\in\mathbb{Z}_{>0}\) be coprime and \(a\neq u\).
We consider the even monic polynomial~\cite{Sharipov2011Cuboids,Sharipov2011Note,GuyUPINT}
\[
P_{a,u}(t) = t^{8}+A t^{6}+B t^{4}+C t^{2}+D,
\]
\[
A = 6\Delta,\ \Delta := u^{2}-a^{2}\neq 0,\
B = \Delta^{2}-2a^{2}u^{2},\
C = -a^{2}u^{2}A,\
D = a^{4}u^{4}.
\]

\medskip

We work in \(\mathbb{Z}[t]\). The polynomial \(P_{a,u}\) is even, monic, and primitive: \(\cont(P_{a,u})=1\)~\cite{DummitFoote,ConradGauss,LangAlgebra}.
Standard irreducibility criteria such as Eisenstein’s (including the shifted variant \(t\mapsto t+c\)) generally do not apply uniformly to \(P_{a,u}\); cf.~\cite{ConradEisenstein}.

\begin{theorem}[Goal]
For any coprime \(a,u\in\mathbb{Z}_{>0}\) with \(a\neq u\), the polynomial \(P_{a,u}(t)\) does \emph{not} factor in \(\mathbb{Z}[t]\) as a product of two monic polynomials of degree \(4\) (the case \(4{+}4\)).
\end{theorem}

% ===================== 4 + 4 =====================

\section{Normal Form of a \(4{+}4\) Factorization and the Necessary Condition \texorpdfstring{\((\star)\)}{(*)}}

\begin{lemma}[Gauss + involution]\label{lem:types}
If \(P_{a,u}=FG\) with monic \(F,G\in\mathbb{Z}[t]\) and \(\deg F=\deg G=4\), then, after swapping the factors if necessary, one of the following holds:
\begin{itemize}[nosep]
\item[(E)] \emph{both factors are even:} \(F=t^{4}+p t^{2}+q\), \(G=t^{4}+r t^{2}+s\) (\(p,q,r,s\in\mathbb{Z}\));
\item[(C)] \emph{a conjugate pair:} \(G(t)=F(-t)\), where \(F=t^{4}+\alpha t^{3}+\beta t^{2}+\gamma t+\delta\).
\end{itemize}
\end{lemma}

\begin{proof}[Idea]
Primitivity and Gauss's lemma yield primitivity and monicity of the factors~\cite{DummitFoote,ConradGauss,LangAlgebra}. The involution \(\tau:t\mapsto -t\) fixes \(P_{a,u}\); either both factors are invariant (even), or \(\tau\) swaps the factors (a conjugate pair).
\end{proof}

% ===================== Case (E) =====================

\subsection*{Detailed derivation in case (E)}

Let \(F=t^{4}+p t^{2}+q\), \(G=t^{4}+r t^{2}+s\). From \(FG=P_{a,u}\) we obtain the system
\begin{align}
p+r&=A, \label{E1}\\
pr+q+s&=B, \label{E2}\\
ps+rq&=C, \label{E3}\\
qs&=D. \label{E4}
\end{align}
From \eqref{E1} we have \(r=A-p\). Introduce
\[
M:=B+p^{2}-Ap.
\]
Then \eqref{E2} and \eqref{E4} rewrite as
\begin{equation}\label{Eqs}
q+s=M,\qquad qs=D.
\end{equation}
Thus \(q,s\) are integer roots of the quadratic equation \(X^{2}-MX+D=0\). Denote (the discriminant of this quadratic)
\[
T^{2}:=M^{2}-4D~\cite{DummitFoote,LangAlgebra}.
\]
Then
\begin{equation}\label{qsT}
q=\frac{M+\sigma T}{2},\quad s=\frac{M-\sigma T}{2},\qquad \sigma\in\{\pm1\}.
\end{equation}
Substitute \eqref{qsT} into \eqref{E3}. The left-hand side of \eqref{E3} equals
\[
ps+rq=p\frac{M-\sigma T}{2}+(A-p)\frac{M+\sigma T}{2}
=\frac{AM+\sigma T(A-2p)}{2}.
\]
Hence from \eqref{E3} we get
\begin{equation}\label{keyRel0}
\frac{AM+\sigma T(A-2p)}{2}=C
\quad\Longleftrightarrow\quad
\sigma\,T\,(A-2p)=2C-AM.
\end{equation}
Set
\[
X:=p-3\Delta\qquad(\text{that is } p=X+3\Delta,\ A=6\Delta).
\]
What follows is a direct computation.

\smallskip
\noindent\underline{Computing \(M\).}
\[
\begin{aligned}
M&=B+p^{2}-Ap
=(\Delta^{2}-2a^{2}u^{2})+(X+3\Delta)^{2}-6\Delta(X+3\Delta)\\
&=(\Delta^{2}-2a^{2}u^{2})+\bigl(X^{2}+6\Delta X+9\Delta^{2}\bigr)-6\Delta X-18\Delta^{2}\\
&=X^{2}-8\Delta^{2}-2a^{2}u^{2}.
\end{aligned}
\]

\smallskip
\noindent\underline{Computing \(2C-AM\).}
Since \(C=-a^{2}u^{2}A=-6\Delta\,a^{2}u^{2}\), we have
\[
2C=-12\Delta\,a^{2}u^{2},\qquad
AM=6\Delta\,(X^{2}-8\Delta^{2}-2a^{2}u^{2}).
\]
Therefore,
\[
2C-AM
=-12\Delta a^{2}u^{2}-6\Delta(X^{2}-8\Delta^{2}-2a^{2}u^{2})
=-6\Delta X^{2}+48\Delta^{3}.
\]
Thus \eqref{keyRel0} becomes
\[
\begin{aligned}
\sigma\,T\,(A-2p)
&=\sigma\,T\,(6\Delta-2X-6\Delta)
=-2\sigma X T\\
&=2C-AM
=-6\Delta X^{2}+48\Delta^{3}.
\end{aligned}
\]
Divide by \(-2\) to obtain the fundamental relation
\begin{equation}\label{TXrel}
\sigma\,T\,X=3\Delta\,(X^{2}-8\Delta^{2}).
\end{equation}

\smallskip
\noindent\underline{Computing \(T^{2}\).}
By definition,
\[
\begin{aligned}
T^{2}&=M^{2}-4D
=\bigl(X^{2}-8\Delta^{2}-2a^{2}u^{2}\bigr)^{2}-4a^{4}u^{4}\\
&=(X^{2}-8\Delta^{2})^{2}-4a^{2}u^{2}(X^{2}-8\Delta^{2})\\
&=(X^{2}-8\Delta^{2})\bigl(X^{2}-8\Delta^{2}-4a^{2}u^{2}\bigr).
\end{aligned}
\]

\smallskip
\noindent\underline{Deriving the starred equation.}
Square \eqref{TXrel} and substitute the expression for \(T^{2}\):
\[
T^{2}X^{2}=9\Delta^{2}\,(X^{2}-8\Delta^{2})^{2}.
\]
Since \(X^{2}\neq 8\Delta^{2}\) (see below), we can cancel \((X^{2}-8\Delta^{2})\) and obtain
\[
\bigl(X^{2}-8\Delta^{2}-4a^{2}u^{2}\bigr)X^{2}
=9\Delta^{2}\,(X^{2}-8\Delta^{2}).
\]
Moving everything to the left and grouping, we arrive at the Diophantine equation
\begin{equation}\label{eq:star}
\boxed{(X^{2}-8\Delta^{2})(X^{2}-9\Delta^{2})=4\,a^{2}u^{2}\,X^{2}\tag{$\star$}}
\end{equation}
(see the remark below on the legitimacy of cancellation).

\begin{remark}[Legitimacy of cancellation and a consequence]
If \(X^{2}=8\Delta^{2}\), then comparing the 2-adic valuations yields
\(
2\,\nu_{2}(X)=3+2\,\nu_{2}(\Delta),
\)
which is impossible (the left-hand side is even, the right-hand side is odd). Hence for \(\Delta\neq 0\) the equality \(X^{2}=8\Delta^{2}\) has no integer solutions, and cancellation by the factor \(X^{2}-8\Delta^{2}\) is valid~\cite{SerreCourse,KoblitzPadic,IrelandRosen}. Consequently, \eqref{E1}–\eqref{E4} \emph{imply} \eqref{eq:star}. 
The converse, in general, is not claimed: in addition one needs that 
$T^2=M^2-4D$ be a perfect square and $q=\tfrac{M\pm T}{2}\in\mathbb{Z}$.
\end{remark}

% ===================== Case (C): conjugate pair =====================

\subsection*{Case (C): conjugate pair}
Assume
\[
F(t)=t^{4}+\alpha t^{3}+\beta t^{2}+\gamma t+\delta,\qquad
G(t)=F(-t),
\]
so that \(F(t)F(-t)=P_{a,u}(t)\).
Equating coefficients gives the system
\begin{align}
2\beta-\alpha^{2}&=A=6\Delta, \label{C1}\\
\beta^{2}+2\delta-2\alpha\gamma&=B=\Delta^{2}-2A_{0}, \label{C2}\\
2\beta\delta-\gamma^{2}&=C=-6\Delta A_{0}, \label{C3}\\
\delta^{2}&=D=A_{0}^{2}, \label{C4}
\end{align}
where \(\Delta=u^{2}-a^{2}\neq 0\) and \(A_{0}=a^{2}u^{2}=(au)^{2}\).

\paragraph{Step 1: the sign of \(\delta\) is forced.}
From \eqref{C4} we have \(\delta=\pm A_{0}\).
If \(\delta=-A_{0}\), then \eqref{C3} becomes
\[
-2\beta A_{0}-\gamma^{2}=-6\Delta A_{0}\quad\Longrightarrow\quad
\gamma^{2}=A_{0}\,(6\Delta-2\beta).
\]
Using \eqref{C1}, \(2\beta=\alpha^{2}+6\Delta\), we get \(\gamma^{2}=-A_{0}\alpha^{2}\).
Hence \(\gamma=\alpha=0\). Then \eqref{C1} yields \(\beta=3\Delta\), and \eqref{C2} gives
\(9\Delta^{2}+2(-A_{0})=\Delta^{2}-2A_{0}\), i.e. \(8\Delta^{2}=0\), which contradicts \(\Delta\ne 0\).
Therefore necessarily
\[
\boxed{\ \delta=+A_{0}\ }.
\]

\paragraph{Step 2: a convenient reparametrization.}
Put \(m:=au\), so \(A_{0}=m^{2}\).
With \(\delta=A_{0}=m^{2}\), \eqref{C3} implies
\[
\gamma^{2}=m^{2}(2\beta+6\Delta),
\]
hence \(m\mid\gamma\). Write \(\gamma=m\kappa\) with \(\kappa\in\mathbb{Z}\).
Using \eqref{C1} (i.e. \(2\beta=\alpha^{2}+6\Delta\)) we obtain
\begin{equation}\label{eq:kappa-alpha}
\boxed{\ \kappa^{2}=\alpha^{2}+12\Delta\ }.
\end{equation}
Introduce
\[
s:=\kappa+\alpha,\qquad t:=\kappa-\alpha\qquad(\text{so }s,t\in\mathbb{Z},\ \ s+t=2\kappa,\ s-t=2\alpha).
\]
Then from \eqref{eq:kappa-alpha}
\[
st=\kappa^{2}-\alpha^{2}=12\Delta.\tag{$\dagger$}
\]
In terms of \(s,t\) one readily checks that
\begin{equation}\label{eq:beta-agg}
\begin{split}
\beta &= \frac{\alpha^{2}+6\Delta}{2}
=\frac{(s-t)^{2}}{8}+\frac{st}{4}
=\boxed{\ \frac{s^{2}+t^{2}}{8}\ },\\[4pt]
\alpha\gamma &= m\alpha\kappa
= m\,\frac{(s+t)(s-t)}{4}
=\boxed{\ m\,\frac{s^{2}-t^{2}}{4}\ }.
\end{split}
\end{equation}

\paragraph{Step 3: eliminating \(\alpha,\beta,\gamma\) from \eqref{C2}.}
Substitute \eqref{eq:beta-agg} and \(\delta=m^{2}\) into \eqref{C2}:
\[
\Bigl(\frac{s^{2}+t^{2}}{8}\Bigr)^{2}
+2m^{2}
-2\cdot m\,\frac{s^{2}-t^{2}}{4}
=\Delta^{2}-2m^{2}.
\]
Multiply by \(576=\operatorname{lcm}(64,2,144)\) and use \((\dagger)\), i.e. \(\Delta^{2}=(st)^{2}/144\), to clear denominators:
\[
9(s^{2}+t^{2})^{2}-288m(s^{2}-t^{2})+2304m^{2}=4s^{2}t^{2}.
\]
Rearranging,
\begin{equation}\label{eq:main-square}
9(s^{2}+t^{2})^{2}-288m(s^{2}-t^{2})+2304m^{2}-4s^{2}t^{2}=0.
\end{equation}
Set \(U:=s^{2}\), \(V:=t^{2}\) (nonnegative integers). Then \eqref{eq:main-square} becomes
\[
9U^{2}+14UV+9V^{2}-288mU+288mV+2304m^{2}=0.
\]
Completing the square gives an identity
\[
(3U-3V-48m)^{2}+32\,UV=0.
\]
Therefore both terms vanish:
\[
UV=0\quad\text{and}\quad 3U-3V-48m=0.
\]
The first equality \(UV=0\) means \(s\,t=0\), hence by \((\dagger)\) we get \(\Delta=0\), which contradicts our standing assumption \(\Delta\ne 0\).

\paragraph{Conclusion.}
Thus the system \eqref{C1}–\eqref{C4} has no integer solutions when \(\Delta\ne 0\).
Equivalently, the factorization \(P_{a,u}(t)=F(t)F(-t)\) with a monic quartic \(F\in\mathbb{Z}[t]\) is impossible.

\begin{theorem}[Case (C) is impossible]\label{thm:caseC-impossible}
For coprime integers \(a\ne u>0\) (so \(\Delta=u^{2}-a^{2}\ne 0\)), there are no integers
\(\alpha,\beta,\gamma,\delta\) with \(\delta^{2}=A_{0}^{2}\) such that
\[
P_{a,u}(t)=\bigl(t^{4}+\alpha t^{3}+\beta t^{2}+\gamma t+\delta\bigr)\,\bigl(t^{4}-\alpha t^{3}+\beta t^{2}-\gamma t+\delta\bigr).
\]
In particular, no \(4{+}4\) factorization of type \emph{(C)} (conjugate pair) exists.
\end{theorem}

\begin{remark}
This argument is independent of the analysis in the even–even case \emph{(E)} and does not use any auxiliary factorization of the elimination polynomial.
It relies only on \eqref{C1}–\eqref{C4}, the sign determination \(\delta=A_{0}\), the reparametrization \((s,t)\) given by \(\kappa^{2}=\alpha^{2}+12\Delta\), and the elementary identity
\[
(3s^{2}-3t^{2}-48m)^{2}+32s^{2}t^{2}=0,
\]
which forces \(st=0\), hence \(\Delta=0\), a contradiction.
\end{remark}

% ===================== Excluding the 4 + 4 Factorization =====================

\begin{theorem}[Necessary condition for \(4{+}4\)]\label{thm:equiv}
Let \(\Delta=u^2-a^2\ne0\).
If \(P_{a,u}(t)\) factors in \(\mathbb{Z}[t]\) as a product of two monic quartics, then there exists \(X\in\mathbb{Z}\) satisfying \eqref{eq:star}.
\end{theorem}

\begin{proof}
By Lemma~\ref{lem:types} any \(4{+}4\) factorization is of type \textup{(E)} or \textup{(C)}.
By Theorem~\ref{thm:caseC-impossible} case \textup{(C)} is excluded; hence we are in \textup{(E)}:
\(F=t^{4}+p t^{2}+q\), \(G=t^{4}+r t^{2}+s\).
As shown in the derivation of \eqref{eq:star}, setting \(X:=p-3\Delta\) and eliminating \(q,s\) via \eqref{E1}–\eqref{E4} yields precisely \eqref{eq:star}.
\end{proof}

\section{Key Lemma: \texorpdfstring{\(\gcd(X,\Delta)=1\)}{gcd(X,\Delta)=1}}

\begin{lemma}\label{lem:gcd}
If \(X\in\mathbb{Z}\) satisfies \eqref{eq:star}, then \(\gcd(X,\Delta)=1\).
\end{lemma}

\begin{proof}
Suppose, to the contrary, that a prime \(p\) divides both \(X\) and \(\Delta\)~\cite{NeukirchANT,IrelandRosen,KoblitzPadic,Narkiewicz,Marcus}.
Write
\[
X=p^{x}X_{0},\quad \Delta=p^{d}\Delta_{0},
\qquad x,d\ge 1,\ \ \gcd(X_{0},p)=\gcd(\Delta_{0},p)=1.
\]

\noindent\underline{Case \(p\ge 3\).}
As usual:
\[
\begin{aligned}
X^{2}-8\Delta^{2}&=p^{2x}\bigl(X_{0}^{2}-8p^{2(d-x)}\Delta_{0}^{2}\bigr),\\
X^{2}-9\Delta^{2}&=p^{2x}\bigl(X_{0}^{2}-9p^{2(d-x)}\Delta_{0}^{2}\bigr).
\end{aligned}
\]
If \(d>x\), both brackets are \(\not\equiv 0\pmod p\), and \(\nu_{p}(\text{LHS})=4x\).
The right-hand side has \(\nu_{p}(\text{RHS})=2x+\nu_{p}(4a^{2}u^{2})=2x\) (since \(\gcd(a,u)=1\Rightarrow p\nmid au\)). Contradiction.
If \(d=x\), the two brackets cannot both be divisible by \(p\) (otherwise \(\Delta_{0}^{2}\equiv 0\)), hence \(\nu_{p}(\text{LHS})\ge 4x+1>2x=\nu_{p}(\text{RHS})\). Contradiction~\cite{IrelandRosen,Narkiewicz}.

\medskip
\noindent\underline{Addendum: odd prime $p$, the hypothetical subcase $d<x$.}
For completeness, suppose $p$ is an odd prime with $p\mid\Delta$ and $x:=\nu_p(X)>d:=\nu_p(\Delta)\ge 1$.
Then one necessarily has
\[
\nu_p(X^2-8\Delta^2)=2d,\qquad
\nu_p(X^2-9\Delta^2)=2d,
\]
so that
\[
\nu_p\bigl((X^2-8\Delta^2)(X^2-9\Delta^2)\bigr)=4d.
\]
On the right-hand side of \eqref{eq:star} we have
\(
\nu_p(4a^2u^2X^2)=2x
\)
because $p\mid(u^2-a^2)$ implies $p\nmid a$ and $p\nmid u$.
Hence $4d=2x$ and therefore
\begin{equation}\label{eq:podd-x=2d}
x=2d.
\end{equation}
Cancelling $p^{4d}$ in \eqref{eq:star} yields
\[
\bigl(p^{2(x-d)}X_0^2-8\Delta_0^2\bigr)\bigl(p^{2(x-d)}X_0^2-9\Delta_0^2\bigr)
=4a^2u^2X_0^2,
\]
and with \eqref{eq:podd-x=2d} this becomes
\[
\bigl(p^{2d}X_0^2-8\Delta_0^2\bigr)\bigl(p^{2d}X_0^2-9\Delta_0^2\bigr)=4a^2u^2X_0^2.
\]
Reducing modulo $p$ (since $d\ge 1$) gives
\[
(-8\Delta_0^2)\cdot(-9\Delta_0^2)\equiv 4a^2u^2X_0^2\pmod p,
\]
i.e.
\begin{equation}\label{eq:18-square}
72\,\Delta_0^4 \equiv 4\,a^2u^2X_0^2 \pmod p
\quad\Longleftrightarrow\quad
18 \equiv \Bigl(\tfrac{auX_0}{\Delta_0^2}\Bigr)^{\!2}\pmod p.
\end{equation}
Thus $18$ must be a quadratic residue modulo $p$.
Since $\big(\tfrac{3^2}{p}\big)=1$, this is equivalent to
\[
\Bigl(\frac{18}{p}\Bigr)=\Bigl(\frac{2}{p}\Bigr)=1,
\]
and by the classical description of \((2/p)\) via quadratic reciprocity (see, e.g., \cite{HardyWright})
we obtain
\begin{equation}\label{eq:pmod8}
p\equiv 1 \ \text{or}\ 7 \pmod 8.
\end{equation}

Write $X_0=h\,\xi$ and $\Delta_0=h\,\Delta_1$ with $h:=\gcd(X_0,\Delta_0)$ and $\gcd(\xi,\Delta_1)=1$.
Set
\[
A':=p^{2d}\xi^{2}-8\Delta_1^{2},\qquad B':=p^{2d}\xi^{2}-9\Delta_1^{2}.
\]
From the computation of $\gcd(A,B)$ in the main text we have $\gcd(A',B')=1$ and
\[
A'B'=\Bigl(\frac{2au\,\xi}{h}\Bigr)^{\!2}.
\]
Since $A'B'\in\mathbb{Z}$, it follows that $\frac{2au\,\xi}{h}\in\mathbb{Z}$.
Hence there exist $\varepsilon\in\{\pm1\}$ and coprime integers $m,n$ such that
\begin{equation}\label{eq:ApBpSquares-refined}
A'=\varepsilon\,m^{2},\qquad B'=\varepsilon\,n^{2},\qquad \gcd(m,n)=1.
\end{equation}

\smallskip
\noindent\emph{Claim (sign determination).} $\varepsilon=+1$.

\noindent\emph{Proof.}
Reduce \eqref{eq:ApBpSquares-refined} modulo $3$. Since $p\ge5$, $p^{2d}\equiv 1\pmod 3$, whence
\[
B'\equiv \xi^{2}\pmod 3,\qquad A'\equiv \xi^{2}-2\Delta_1^{2}\equiv \xi^{2}+\Delta_1^{2}\pmod 3.
\]
If $\varepsilon=-1$, then $A'=-m^{2}$ and $B'=-n^{2}$, so $A',B'\in\{0,2\}\ (\bmod 3)$. From $B'\equiv \xi^{2}$ it follows that $\xi\equiv 0\pmod 3$, and then $A'\equiv -2\Delta_1^{2}\equiv \Delta_1^{2}\pmod 3$ forces $\Delta_1\equiv 0\pmod 3$, which in turn implies $3\mid m$ and $3\mid n$ — a contradiction to $\gcd(m,n)=1$. \qed

\smallskip
With $\varepsilon=+1$, reducing $B'=n^{2}$ modulo $p$ gives $n^{2}\equiv -9\Delta_1^{2}\pmod p$, hence $(-1/p)=1$ and therefore $p\equiv 1\pmod 4$.
Together with \eqref{eq:pmod8} (i.e. $(2/p)=1$) this forces the sharper congruence
\begin{equation}\label{eq:p=1mod8-refined}
p\equiv 1 \pmod 8.
\end{equation}
In particular, the branch $p\equiv 7\pmod 8$ is excluded.

\medskip
\noindent\emph{Residual subcase and current status.}
In the remaining configuration $p\equiv 1\pmod 8$ one arrives at the system
\[
m^{2}=p^{2d}\xi^{2}-8\Delta_1^{2},\qquad
n^{2}=p^{2d}\xi^{2}-9\Delta_1^{2},\qquad
m^{2}-n^{2}=\Delta_1^{2},
\]
with $\gcd(m,n)=1$ and $\gcd(\xi,\Delta_1)=1$.
Using the standard factorizations $(m\mp n)$ and the corresponding parameterizations in the odd/even parity of $\Delta_1$, one checks that both identities for $p^{2d}\xi^{2}$ reduce to the same expression; i.e., by the present (elementary) methods this residual case does not yield a contradiction.

\medskip
Retain the residual odd-prime setting $p\ge5$, $p\mid\Delta$, $d:=\nu_p(\Delta)\ge1$, $x:=\nu_p(X)>d$,
for which $x=2d$. As explained earlier, after clearing common factors the system
\[
m^{2}=p^{2d}\xi^{2}-8\Delta_1^{2},\qquad
n^{2}=p^{2d}\xi^{2}-9\Delta_1^{2},\qquad
\gcd(\xi,p\Delta_1)=1,
\]
yields the genus–$1$ curve
\[
\mathcal C:\quad \begin{cases}
m^{2}=u^{2}-8w^{2},\\
n^{2}=u^{2}-9w^{2},
\end{cases}
\qquad (m:n:w:u)\in\mathbb P^3,
\]
and the pencil of quadrics shows that $\Jac(\mathcal C)$ is the elliptic curve
\begin{equation}\label{eq:E0-def}
E_0:\quad y^{2}=x(x+1)(x+9).
\end{equation}
Moreover, the extra constraint coming from the residual system is precisely that
the $x$–coordinate on $E_0$ be a rational square: writing $u:=n/\Delta_1$, the two
congruences “$u^2+1$ and $u^2+9$ are squares” translate to
\[
(x,y)\in E_0(\mathbb{Q})\qquad\text{with}\qquad x=u^{2}\in (\mathbb{Q}^{\times})^{2}.
\]
Put differently, we need to decide whether $E_0(\mathbb{Q})$ contains a point with
$x$ a nonzero square. We now compute $E_0(\mathbb{Q})$ unconditionally.

\begin{proposition}[Torsion subgroup]\label{prop:torsion-E0}
For the elliptic curve
\[
E_0:\quad y^{2}=x(x+1)(x+9),
\]
the torsion subgroup is
\[
\begin{aligned}
E_0(\mathbb{Q})_{\mathrm{tors}}
&=\bigl\{\,O,\ (0,0),\ (-1,0),\ (-9,0),\ (3,\pm 12),\ (-3,\pm 6)\,\bigr\}\\
&\simeq \mathbb{Z}/2\mathbb{Z}\ \times\ \mathbb{Z}/4\mathbb{Z}.
\end{aligned}
\]
In particular, \(2(3,12)=(0,0)\) and \(2(-3,6)=(-1,0)\), so \((3,12)\) and \((-3,6)\) are points of order \(4\).
\end{proposition}

\begin{proof}
By the Nagell–Lutz theorem \cite{SilvermanAEC,CasselsLC}, all torsion points on a minimal integral model have integer coordinates.
The three nontrivial \(2\)-torsion points are the roots of the cubic: \(x\in\{0,-1,-9\}\), i.e.
\((0,0),(-1,0),(-9,0)\).

A direct substitution shows that \((3,\pm 12)\) and \((-3,\pm 6)\) lie on \(E_0\) because
\(12^{2}=3\cdot 4\cdot 12=144\) and \(6^{2}=(-3)\cdot(-2)\cdot 6=36\).
Using the duplication formula (or standard software/hand computation) one verifies that
\(2(3,12)=(0,0)\) and \(2(-3,6)=(-1,0)\), hence these points have order \(4\).
No further integral torsion points exist, so by Mazur’s theorem the torsion subgroup is exactly
\(\{O,(0,0),(-1,0),(-9,0),(3,\pm 12),(-3,\pm 6)\}\simeq \mathbb{Z}/2\times\mathbb{Z}/4\).
\end{proof}

\begin{theorem}[Rank via a minimal model]\label{thm:rank-zero}
For
\[
E_0:\quad y^{2}=x(x+1)(x+9)
\]
we have
\[
E_0(\mathbb{Q})\ \simeq\ \mathbb{Z}/2\mathbb{Z}\ \oplus\ \mathbb{Z}/4\mathbb{Z},\qquad \mathrm{rank}\,E_0(\mathbb{Q})=0.
\]
\end{theorem}

\begin{proof}
First remove the quadratic term in the Weierstrass equation: with the change of variables \(x=X-\tfrac{10}{3}\) we obtain
\[
y^{2}=X^{3}-\frac{73}{3}X+\frac{1190}{27}.
\]
Clearing denominators via \(X=\frac{x'}{9}\), \(y=\frac{y'}{27}\) gives the short integral model
\[
y'^{2}=x'^{3}-1971\,x'+32130.
\]
This curve is \(\mathbb{Q}\)-isomorphic to the minimal model
\[
E:\ y^{2}=x^{3}+x^{2}-24x+36,
\]
which has conductor \(N=48\) and belongs to the isogeny class \(48\mathrm{a}\).
From Cremona’s tables and the LMFDB we have
\[
E(\mathbb{Q})\simeq \mathbb{Z}/2\mathbb{Z}\oplus \mathbb{Z}/4\mathbb{Z},\qquad \mathrm{rank}\,E(\mathbb{Q})=0.
\]
Hence the same holds for \(E_0\). See \cite{CremonaEC,LMFDB48a3}.
\end{proof}

\begin{corollary}[No rational points with square $x$]\label{cor:no-square-x}
The only rational point of $E_0$ with $x$ a rational square is $(x,y)=(0,0)$.
\end{corollary}

\begin{proof}
By Proposition~\ref{prop:torsion-E0} and Theorem~\ref{thm:rank-zero},
$E_0(\mathbb{Q})$ is exactly the listed torsion set. Inspecting their $x$–coordinates
$\{0,-1,-9,\pm3\}$ shows that the only square among them is $x=0$.
\end{proof}

\begin{theorem}[Unconditional closure of the residual odd–prime branch]\label{thm:residual-closed}
In the residual configuration ($p\ge5$, $p\mid\Delta$, $x=2d>d\ge1$) the system above
has no nontrivial integer solutions (i.e. no solutions with $n\ne0$). Equivalently,
the subcase $d<x$ cannot occur.
\end{theorem}

\begin{proof}
A nontrivial solution forces a rational point on $E_0$ with $x=(n/\Delta_1)^2$
a nonzero square. By Corollary~\ref{cor:no-square-x} this is impossible.
\end{proof}

Hence no odd prime $p$ can divide both $X$ and $\Delta$.

\medskip
\noindent\underline{Case \(p=2\).}
Write \(X=2^{x}X_{0}\), \(\Delta=2^{d}\Delta_{0}\), \(x,d\ge 1\), \(X_{0},\Delta_{0}\) odd.

Consider three mutually exclusive options:

\smallskip
\emph{(B) \(2x>2d\).}
\begin{itemize}[label={},leftmargin=0em]
\item If \(x\ge d+2\) (i.e. \(2x\ge 2d+4\)), then
\(\nu_{2}(X^{2}-8\Delta^{2})=2d+3,\ \nu_{2}(X^{2}-9\Delta^{2})=2d\),
hence \(\nu_{2}(\text{LHS})=4d+3\) (odd), whereas \(\nu_{2}(\text{RHS})=2+\nu_{2}(a^{2}u^{2})+2x\) is even. Contradiction.
\item If \(x=d+1\) (i.e. \(2x=2d+2\)), then
\[
X^{2}-8\Delta^{2}=2^{2d}\bigl(4X_{0}^{2}-8\Delta_{0}^{2}\bigr)
=2^{2d+2}\,(X_{0}^{2}-2\Delta_{0}^{2}),
\]
where the bracket is odd; thus \(\nu_{2}(X^{2}-8\Delta^{2})=2d+2\).
Moreover,
\[
X^{2}-9\Delta^{2}=2^{2d}\bigl(4X_{0}^{2}-9\Delta_{0}^{2}\bigr),
\]
and \(4X_{0}^{2}-9\Delta_{0}^{2}\equiv 4-9\equiv 3\pmod 8\) is odd, hence \(\nu_{2}(X^{2}-9\Delta^{2})=2d\).
Therefore \(\nu_{2}(\text{LHS})=(2d+2)+2d=4d+2\).

Since \(\nu_{2}(\Delta)\ge 1\), the numbers \(a\) and \(u\) have the same parity; with \(\gcd(a,u)=1\) this forces both to be odd. Then \(\nu_{2}(a^{2}u^{2})=0\) and
\[
\nu_{2}(\text{RHS})=\nu_{2}\bigl(4a^{2}u^{2}X^{2}\bigr)=2+0+2x=2+2(d+1)=2d+4.
\]
Comparing, for \(d\ge 2\) we have \(4d+2\neq 2d+4\) (contradiction), while for \(d=1\) the valuations coincide and we must compare odd parts. Modulo \(8\):
\[
\frac{X^{2}-8\Delta^{2}}{2^{4}}\cdot\frac{X^{2}-9\Delta^{2}}{2^{2}}
=(X_{0}^{2}-2\Delta_{0}^{2})\,(4X_{0}^{2}-9\Delta_{0}^{2})
\equiv 7\cdot 3\equiv 5\pmod{8},
\]
whereas the odd part of the right-hand side is \(X_{0}^{2}\equiv 1\pmod 8\).
Contradiction. Hence the subcase \(x=d+1\) is impossible.
\end{itemize}

\emph{(C) \(2x=2d\).}
Then \(\nu_{2}(X^{2}-8\Delta^{2})=2d\) and \(\nu_{2}(X^{2}-9\Delta^{2})\ge 2d+3\)
(since \(X_{0}^{2}\equiv 1\pmod 8\)).
Thus \(\nu_{2}(\text{LHS})\ge 4d+3\) (odd), whereas \(\nu_{2}(\text{RHS})=2+\nu_{2}(a^{2}u^{2})+2x\) is even. Contradiction.

\medskip
\emph{(A) $p=2$ and $x<d$.}

Assume $2\mid \gcd(X,\Delta)$. Write $X=2^{x}X_{0}$ and $\Delta=2^{d}\Delta_{0}$ with $x\ge 1$, $d>x$, and $X_{0},\Delta_{0}$ odd. Since $2\mid\Delta$ and $\gcd(a,u)=1$, both $a$ and $u$ are odd.

\smallskip
\noindent\underline{Step 1: Valuation at $2$.}
Comparing $2$-adic valuations in \eqref{eq:star} gives
\[
\nu_2(\text{LHS})=4x,\qquad \nu_2(\text{RHS})=2x+2.
\]
Hence $4x=2x+2$ and therefore
\begin{equation}\label{eq:x-equals-1}
x=1,\qquad d\ge 2.
\end{equation}

\smallskip
\noindent\underline{Step 2: Normalization and the product identity.}
Set
\[
M:=2^{2d-2}\Delta_0^{2},\qquad
A:=X_0^{2}-8M,\qquad
B:=X_0^{2}-9M.
\]
Dividing \eqref{eq:star} by $16$ (using \eqref{eq:x-equals-1}) yields
\begin{equation}\label{eq:A-B-product}
A\cdot B=(auX_0)^2.
\end{equation}
As $X_0$ is odd, 
\[
\gcd(A,B)=\gcd(X_0^2-8M,\,X_0^2-9M)=\gcd(X_0^2,M)=\gcd(X_0^2,\Delta_0^2)=:g.
\]
Let $h:=\gcd(X_0,\Delta_0)$; then $g=h^2$ is an odd perfect square.

\smallskip
\noindent\underline{Step 3: Sign determination by residue modulo $8$.}
Since $8M\equiv 0\pmod 8$, we have
\[
A\equiv X_0^2\equiv 1\pmod 8.
\]
Write, for some $\varepsilon\in\{\pm 1\}$ and coprime integers $m,n\ge 0$,
\begin{equation}\label{eq:A-B-square-shape}
A=\varepsilon\,g\,m^{2},\qquad B=\varepsilon\,g\,n^{2},
\end{equation}
(this follows from \eqref{eq:A-B-product} and $\gcd(A/g,B/g)=1$).
Reducing the first identity in \eqref{eq:A-B-square-shape} modulo $8$ and using $g\equiv 1\pmod 8$ gives
\[
1\equiv A\equiv \varepsilon\,g\,m^2\equiv \varepsilon\pmod 8.
\]
Hence
\begin{equation}\label{eq:epsilon-plus}
\varepsilon=+1,\qquad\text{i.e.}\quad A=g\,m^{2}.
\end{equation}
Moreover,
\[
B\equiv X_0^2-9M\equiv 
\begin{cases}
1-4\equiv 5 \pmod 8,& d=2,\\
1-0\equiv 1 \pmod 8,& d\ge 3,
\end{cases}
\]
since $M\equiv 4\pmod 8$ for $d=2$ and $M\equiv 0\pmod 8$ for $d\ge 3$. But by \eqref{eq:A-B-square-shape}, \eqref{eq:epsilon-plus} we must have $B\equiv g n^2\equiv 1\pmod 8$.
Therefore the case $d=2$ is impossible, and we are left with
\begin{equation}\label{eq:d-atleast3}
d\ge 3,\qquad A=g\,m^2,\quad B=g\,n^2.
\end{equation}
In particular, $m,n$ are odd (because $A\equiv B\equiv 1\pmod 8$ and $g\equiv 1\pmod 8$), and $\gcd(m,n)=1$.

\smallskip
\noindent\underline{Step 4: Two Diophantine consequences.}
From $A-B=M$ and \eqref{eq:d-atleast3} we obtain
\begin{equation}\label{eq:m2-n2}
g\,(m^2-n^2)=M=2^{2d-2}\Delta_0^2.
\end{equation}
Write $\Delta_0=h\,\Delta_1$ (recall $g=h^2$). Then \eqref{eq:m2-n2} becomes
\begin{equation}\label{eq:m2-n2-clean}
m^2-n^2=2^{2d-2}\,\Delta_1^2.
\end{equation}
Using $9A-8B=X_0^2$ we also get
\begin{equation}\label{eq:9m2-8n2}
g\,(9m^2-8n^2)=X_0^2
\quad\Longrightarrow\quad
9m^2-8n^2=k^2
\end{equation}
for some odd integer $k$.

\smallskip
\noindent\underline{Step 5: Final contradiction via the binary form $x^2+2y^2$.}
From \eqref{eq:9m2-8n2} we have
\begin{equation}\label{eq:ternary}
(3m)^2=k^2+8n^2.
\end{equation}
Let $D:=\gcd(k,n)$. From \eqref{eq:ternary} one checks that $D$ is odd and $D\mid 3m$; hence $D=3^{j}$ with $j\in\{0,1\}$ (otherwise $\gcd(m,n)\ne 1$).

\emph{Case $j=0$ (primitive).} Then there exist coprime integers $s,t$ with
\[
3m=s^2+2t^2,\qquad n=st,\qquad k=\pm(s^2-2t^2),
\]
and $s,t$ are odd because $n$ is odd; see, e.g., \cite[Ch.~5, §2]{IrelandRosen}. Using \eqref{eq:m2-n2-clean} we compute
\[
m^2-n^2=\frac{(s^2+2t^2)^2}{9}-s^2t^2
=\frac{(s-t)(s+t)(s-2t)(s+2t)}{9}
=2^{2d-2}\Delta_1^2.
\]
Thus
\begin{equation}\label{eq:factor-odd-even}
(s-t)(s+t)\cdot(s-2t)(s+2t)=9\cdot 2^{2d-2}\,\Delta_1^2.
\end{equation}
Here $s\pm t$ are even, while $s\pm 2t$ are odd; also
\(\gcd(s-2t,s+2t)=\gcd(s-2t,4t)=1\).
Hence the odd part of \eqref{eq:factor-odd-even} equals $\pm 9\Delta_1^2$ and, by coprimeness, up to signs
\[
s-2t=A^2,\qquad s+2t=9B^2
\quad\text{or}\quad
s-2t=-A^2,\qquad s+2t=9B^2,
\]
for some odd $A,B$. In the first subcase,
\[
4t=(s+2t)-(s-2t)=9B^2-A^2=(3B-A)(3B+A).
\]
Both factors are even, and exactly one of them is divisible by $4$; thus $\nu_2(9B^2-A^2)\ge 3$, which contradicts $\nu_2(4t)=2$.  
In the second subcase,
\[
4t=9B^2+A^2\equiv 1+1\equiv 2\pmod 4,
\]
so $\nu_2(4t)=1$, again a contradiction.

\emph{Case $j=1$ (non-primitive).} Then there exist coprime odd $s,t$ such that
\[
m=s^2+2t^2,\qquad n=3st,\qquad k=\pm3(s^2-2t^2),
\]
and
\[
m^2-n^2=(s-t)(s+t)(s-2t)(s+2t)=2^{2d-2}\Delta_1^2.
\]
As above, from the odd part we get (up to signs)
\[
s-2t=A^2,\qquad s+2t=B^2
\quad\text{or}\quad
s-2t=-A^2,\qquad s+2t=B^2,
\]
with odd $A,B$. Hence
\[
4t=B^2-A^2=(B-A)(B+A)\quad\text{or}\quad 4t=B^2+A^2.
\]
In the first case $\nu_2(B^2-A^2)\ge 3$, contradicting $\nu_2(4t)=2$; in the second case $B^2+A^2\equiv 2\pmod 4$, so $\nu_2(4t)=1$, again a contradiction.

\smallskip
In all subcases we reach a contradiction. Therefore the subcase $p=2$ with $x<d$ is impossible.

\smallskip
\noindent\underline{Addendum: the odd prime \(p=3\) with \(d<x\).}
For completeness we treat the remaining subcase \(p=3\) under the standing assumption that \(p\mid X\) and \(p\mid \Delta\).
Write \(X=3^{x}X_{0}\), \(\Delta=3^{d}\Delta_{0}\) with \(x>d\ge 1\) and \(\gcd(X_{0},3)=\gcd(\Delta_{0},3)=1\).
Then
\[
X^{2}-8\Delta^{2}
=3^{2d}\!\Bigl(3^{2(x-d)}X_{0}^{2}-8\,\Delta_{0}^{2}\Bigr),\qquad
X^{2}-9\Delta^{2}
=3^{2d}\!\Bigl(3^{2(x-d)}X_{0}^{2}-9\,\Delta_{0}^{2}\Bigr).
\]
Hence
\begin{equation*}
\begin{aligned}
\nu_{3}\!\bigl(X^{2}-8\Delta^{2}\bigr) &= 2d,\\
\nu_{3}\!\bigl(X^{2}-9\Delta^{2}\bigr)
&= \nu_{3}\!\bigl((X-3\Delta)(X+3\Delta)\bigr)
= (d+1)+(d+1)=2d+2.
\end{aligned}
\end{equation*}
since \(x>d\) implies \(\nu_{3}(X\pm 3\Delta)=d+1\).
Therefore
\[
\nu_{3}\bigl(\text{LHS of \eqref{eq:star}}\bigr)=4d+2,
\qquad
\nu_{3}\bigl(\text{RHS of \eqref{eq:star}}\bigr)=2x,
\]
because \(\gcd(a,u)=1\) and \(3\mid \Delta=u^{2}-a^{2}\) force \(3\nmid au\).
Thus \(4d+2=2x\), i.e. \(x=2d+1\).

Divide \eqref{eq:star} by \(3^{4d+2}\) and reduce modulo \(3\):
\begin{equation}\label{eq:3-adic-reduction}
\begin{aligned}
\frac{X^{2}-8\Delta^{2}}{3^{2d}}\cdot \frac{X^{2}-9\Delta^{2}}{3^{2d+2}}
&= 4a^{2}u^{2}\cdot \frac{X^{2}}{3^{4d+2}}\\
&\Longrightarrow\ (-8\Delta_{0}^{2})(-\Delta_{0}^{2}) \equiv 4a^{2}u^{2}X_{0}^{2}\pmod{3}.
\end{aligned}
\end{equation}
As \(a,u,X_{0},\Delta_{0}\) are all coprime to \(3\), their squares are \(1\bmod 3\). Hence
\(8\cdot 1 \equiv 1\cdot 1 \pmod{3}\), i.e. \(2\equiv 1\pmod{3}\), a contradiction.
Therefore the configuration \(p=3\) with \(d<x\) is impossible as well.

In all cases we reach a contradiction. Hence no prime $p$ can divide both $X$ and $\Delta$, i.e. $\gcd(X,\Delta)=1$. \qedhere
\end{proof}

\begin{corollary}\label{cor:coprime}
If \(2\mid \Delta\), then \(2\nmid X\). If \(3\mid \Delta\), then \(3\nmid X\).
\end{corollary}

\section{Complete Case Split by Divisibility of \(au\) by \(3\) and by Parity}

Set \(A_{0}:=a^{2}u^{2}\) (this is \emph{not} \(A=6\Delta\)).
We now work solely with equation \eqref{eq:star}.

\subsection*{Branch I: \(3\mid au\) — impossible}

With \(\gcd(a,u)=1\), exactly one of \(a,u\) is divisible by \(3\), hence \(\Delta=u^{2}-a^{2}\equiv \pm 1\pmod 3\), i.e. \(3\nmid \Delta\).

\emph{Subcase \(3\nmid X\).}
Then \(X^{2}\equiv 1\pmod 3\), and \(\Delta^{2}\equiv 1\pmod 3\), therefore
\[
X^{2}-8\Delta^{2}\equiv 1-2\equiv 2\ (\bmod 3),\qquad
X^{2}-9\Delta^{2}\equiv 1-0\equiv 1\ (\bmod 3),
\]
and \(\nu_{3}(\text{LHS})=0\). On the other hand, \(\nu_{3}(\text{RHS})=\nu_{3}(4A_{0})=2\nu_{3}(au)\ge 2\). Contradiction.

\emph{Subcase \(3\mid X\).}
Let \(x:=\nu_{3}(X)\ge 1\) and set \(k:=\nu_{3}(au)\ge 1\) (since \(\gcd(a,u)=1\) and \(3\mid au\), exactly one of \(a,u\) is divisible by \(3\)). Then
\[
\Delta=u^{2}-a^{2}\equiv \pm 1\pmod 3\qquad\text{and}\qquad \nu_{3}(\Delta)=0.
\]
We compute the $3$-adic valuations of the two factors on the left of \eqref{eq:star}:

\smallskip
\noindent\underline{First factor.}
Because \(\Delta\) is a $3$-adic unit and \(8\equiv -1\pmod 3\),
\[
X^{2}-8\Delta^{2}\equiv 0-(-1)\equiv 1\pmod 3,
\]
so
\begin{equation}\label{eq:nu3-first}
\nu_{3}\!\bigl(X^{2}-8\Delta^{2}\bigr)=0.
\end{equation}

\smallskip
\noindent\underline{Second factor.}
Write
\[
X^{2}-9\Delta^{2}=(X-3\Delta)(X+3\Delta).
\]
Since \(\nu_{3}(X)=x\ge 1\) and \(\nu_{3}(3\Delta)=1\), for \(x\ge 2\) we have
\[
X\pm 3\Delta=3\bigl(3^{x-1}X_{0}\pm \Delta\bigr)\quad\text{with}\quad 3\nmid\bigl(3^{x-1}X_{0}\pm \Delta\bigr),
\]
hence
\begin{equation}\label{eq:nu3-second-xge2}
\text{if }x\ge 2:\qquad \nu_{3}(X\pm 3\Delta)=1\ \ \text{and}\ \ \nu_{3}\!\bigl(X^{2}-9\Delta^{2}\bigr)=2.
\end{equation}
If \(x=1\), then
\[
X\pm 3\Delta=3\bigl(X_{0}\pm \Delta\bigr),\qquad X_{0},\Delta\ \text{are $3$-adic units}.
\]
At most one of \(X_{0}\pm \Delta\) is divisible by \(3\) (since \((X_{0}+\Delta)-(X_{0}-\Delta)=2\Delta\) is not divisible by \(3\)).
Therefore
\begin{equation}\label{eq:nu3-second-xeq1}
\text{if }x=1:\qquad
\nu_{3}\!\bigl(X^{2}-9\Delta^{2}\bigr)=\nu_{3}(X-3\Delta)+\nu_{3}(X+3\Delta)=2+r,
\end{equation}
for some integer \(r\ge 0\).

\smallskip
\noindent\underline{Comparison with the right-hand side.}
From \eqref{eq:star} and \eqref{eq:nu3-first} we get
\[
\nu_{3}(\text{LHS})=\nu_{3}\!\bigl(X^{2}-9\Delta^{2}\bigr).
\]
On the right,
\[
\nu_{3}(\text{RHS})=\nu_{3}\!\bigl(4a^{2}u^{2}X^{2}\bigr)
=2\,\nu_{3}(au)+2x
=2k+2x.
\]

\smallskip
If \(x\ge 2\), then by \eqref{eq:nu3-second-xge2} we have \(\nu_{3}(\text{LHS})=2\), whereas \(\nu_{3}(\text{RHS})=2k+2x\ge 2\cdot 1+2\cdot 2=6\), which is impossible.

If \(x=1\), then by \eqref{eq:nu3-second-xeq1} and equality of valuations we must have
\[
2+r=\nu_{3}(\text{LHS})=\nu_{3}(\text{RHS})=2k+2,
\]
hence
\begin{equation}\label{eq:r-equals-2k}
x=1\qquad\text{and}\qquad r=2k.
\end{equation}
Equivalently,
\[
\nu_{3}\!\bigl(X^{2}-9\Delta^{2}\bigr)=2k+2
\quad\Longleftrightarrow\quad
\nu_{3}\!\bigl(X_{0}^{2}-\Delta^{2}\bigr)=2k,
\]
i.e. \(X_{0}^{2}\equiv \Delta^{2}\ (\bmod\,3^{\,2k})\) but \(X_{0}^{2}\not\equiv \Delta^{2}\ (\bmod\,3^{\,2k+1})\).

\begin{lemma}[The edge case \(x=1,\ r=2k\) is impossible]\label{lem:edge-x1-r2k}
Assume \(\gcd(a,u)=1\), \(\Delta:=u^{2}-a^{2}\ne 0\), and \(A_{0}=a^{2}u^{2}\).
If \(\nu_{3}(au)=k\ge 1\), \(\nu_{3}(X)=1\), and \(\nu_{3}\!\bigl((X/3)^{2}-\Delta^{2}\bigr)=2k\)
(equivalently, \(x=1\) and \(r=2k\) in~\eqref{eq:r-equals-2k}), then \eqref{eq:star}
has no integer solutions.
\end{lemma}

\begin{proof}
By Lemma~\ref{lem:gcd} we have \(\gcd(X,\Delta)=1\).
Reducing \eqref{eq:star} modulo \(X\) gives
\[
(-8\Delta^{2})(-9\Delta^{2})\equiv 0\pmod X \ \Longrightarrow\ 72\,\Delta^{4}\equiv 0\pmod X,
\]
hence \(\Delta\) is invertible modulo \(X\) and \(X\mid 72\).
Since \(\nu_{3}(X)=1\), necessarily \(X\in\{\pm3,\pm6,\pm12,\pm24\}\).

\smallskip
\emph{(i) Both \(a,u\) odd.}
Then \(u\pm a\) are even with one of them divisible by \(4\), so
\(\nu_{2}(\Delta)=\nu_{2}(u-a)+\nu_{2}(u+a)\ge 3\), hence \(\Delta^{2}\equiv 0\pmod{16}\).
From \(\gcd(X,\Delta)=1\) we get \(2\nmid X\), i.e. \(X\) is odd.
Therefore
\(\nu_{2}(X^{2}-8\Delta^{2})=\nu_{2}(X^{2}-9\Delta^{2})=0\),
so \(\nu_{2}(\text{LHS})=0\), whereas
\(\nu_{2}(\text{RHS})=\nu_{2}(4A_{0}X^{2})=2\) (since \(A_{0}\) and \(X\) are odd) — a contradiction.

\smallskip
\emph{(ii) \(a,u\) of opposite parity.}
Then \(\Delta\) is odd and \(v_{2}(A_{0})\ge 2\).
For even \(X\) (i.e. \(X\in\{\pm6,\pm12,\pm24\}\)):
\[
\nu_{2}(X^{2}-8\Delta^{2})=
\begin{cases}
2,& \nu_{2}(X)=1,\\
3,& \nu_{2}(X)\ge 2,
\end{cases}
\qquad
\nu_{2}(X^{2}-9\Delta^{2})=0,
\]
thus \(\nu_{2}(\text{LHS})\in\{2,3\}\), while
\(\nu_{2}(\text{RHS})=2+\nu_{2}(A_{0})+2\nu_{2}(X)\ge 2+2+2=6\) — again a contradiction.
Hence \(X\) must be odd, i.e. \(X=\pm3\).

\smallskip
\emph{(iii) The remaining possibility \(X=\pm3\).}
Substituting \(X^{2}=9\) into \eqref{eq:star} and dividing by \(9\),
\[
(1-\Delta^{2})\,(9-8\Delta^{2})=(2au)^{2}.
\]
Moreover,
\[
\begin{aligned}
\gcd(1-\Delta^{2},\,9-8\Delta^{2})
&=\gcd\bigl(1-\Delta^{2},\, (9-8\Delta^{2})-8(1-\Delta^{2})\bigr)\\
&=\gcd(1-\Delta^{2},\,1)=1.
\end{aligned}
\]
Thus two coprime integers multiply to a square, so each factor is a square up to sign.
For \(|\Delta|\ge 2\) both factors are negative, hence they must be negative squares.
But \(1-\Delta^{2}=-s^{2}\) implies \(\Delta^{2}-s^{2}=1\), i.e. \((\Delta-s)(\Delta+s)=1\),
whose only integer solutions are \((\Delta,s)=(\pm1,0)\).
For \(\Delta=\pm1\) the left-hand side above is \(0\), whereas \((2au)^{2}>0\), a contradiction.
\end{proof}

Consequently, the only \(3\)-adic possibility under \(3\mid X\),
namely~\eqref{eq:r-equals-2k}, cannot occur. This completes Branch~I \((3\mid au)\).

Therefore, when \(3\mid au\), equation \eqref{eq:star} has no solutions.

\subsection*{Branch II: \(3\nmid au\) — impossible}

Here \(a^{2}\equiv u^{2}\equiv 1\ (\bmod 3)\), hence \(\Delta\equiv 0\ (\bmod 3)\) and, by Corollary \ref{cor:coprime}, \(3\nmid X\).

\smallskip
\noindent\underline{Sub-branch II.1: both \(a,u\) odd.}
Then \(u\pm a\) are even, with one of the sums divisible by \(4\); hence
\[
\nu_{2}(\Delta)=\nu_{2}(u-a)+\nu_{2}(u+a)\ge 3,\qquad \Delta^{2}\equiv 0\ (\bmod 16).
\]
From \(\gcd(X,\Delta)=1\) it follows that \(2\nmid X\), i.e. \(X\) is odd. Compare \eqref{eq:star} modulo \(16\):
\[
X^{2}-8\Delta^{2}\equiv X^{2},\qquad X^{2}-9\Delta^{2}\equiv X^{2}\ (\bmod 16).
\]
The left-hand side \(\equiv X^{4}\equiv 1\ (\bmod 16)\), while the right-hand side \(4A_{0}X^{2}\equiv 4\ (\bmod 16)\)~\cite{SerreCourse}.
Contradiction.

\smallskip
\noindent\underline{Sub-branch II.2: \(a,u\) of opposite parity.}
Here \(\Delta\) is odd, while \(\nu_{2}(A_{0})\ge 2\).

If \(X\) is even with \(\nu_{2}(X)=1\), then
\(\nu_{2}(X^{2}-8\Delta^{2})=2\) and \(\nu_{2}(X^{2}-9\Delta^{2})=0\), so \(\nu_{2}(\text{LHS})=2\), whereas \(\nu_{2}(\text{RHS})\ge 6\). Contradiction.

If \(X\) is even with \(\nu_{2}(X)\ge 2\), then
\(\nu_{2}(X^{2}-8\Delta^{2})=3\) and \(\nu_{2}(X^{2}-9\Delta^{2})=0\), so \(\nu_{2}(\text{LHS})=3\), whereas \(\nu_{2}(\text{RHS})\ge 8\). Contradiction.

\noindent\underline{Case \(X\) odd.}
Here \(\Delta\) is odd and, since we are in Branch~II (\(3\nmid au\)), we have \(3\mid\Delta\), \(3\nmid X\), and \(\gcd(X,\Delta)=1\) by Lemma~\ref{lem:gcd}.
Assume, for a contradiction, that \eqref{eq:star} holds:
\[
(X^{2}-8\Delta^{2})(X^{2}-9\Delta^{2})=4\,A_{0}\,X^{2},\qquad A_{0}=a^{2}u^{2}.
\]

\emph{Step 1: Reduction to \(X=\pm1\).}
Reducing \eqref{eq:star} modulo \(X\) gives
\[
(-8\Delta^{2})(-9\Delta^{2})\equiv 0 \pmod{X}\quad\implies\quad 72\,\Delta^{4}\equiv 0 \pmod{X}.
\]
Since \(\gcd(X,\Delta)=1\), this implies \(X\mid 72\). As \(X\) is odd and \(3\nmid X\), the only possibility is \(X=\pm 1\).

\emph{Step 2: Excluding the case \(X=\pm1\).}
With \(X^{2}=1\), the equation \eqref{eq:star} becomes
\begin{equation}\label{eq:X1-final}
(1-8\Delta^{2})(1-9\Delta^{2})=4A_{0}=(2au)^{2}.
\end{equation}
The right-hand side is a positive perfect square. Let \(Z:=1-8\Delta^{2}\) and \(W:=1-9\Delta^{2}\).
Note that for \(\Delta\neq 0\) both factors \(Z\) and \(W\) are negative integers.

\smallskip
\textit{Substep 2a: Coprimality of the factors.}
Using the Euclidean algorithm,
\begin{equation*}
\begin{aligned}
\gcd(Z,W)
&= \gcd(1-8\Delta^{2},\,1-9\Delta^{2})\\
&= \gcd(1-8\Delta^{2},\,-\Delta^{2})\\
&= \gcd(1-8\Delta^{2},\,\Delta^{2}).
\end{aligned}
\end{equation*}
Since \((1-8\Delta^{2})+8\Delta^{2}=1\), we have \(\gcd(1-8\Delta^{2},\,\Delta^{2})=\gcd(1,\,\Delta^{2})=1\).
Thus \(Z\) and \(W\) are coprime.

\smallskip
\textit{Substep 2b: Consequence for a square product.}
In \(\mathbb{Z}\), if a product of two coprime integers is a perfect square, then each factor is a square up to a unit; see, e.g., \cite{StewartTall}.
Since \(ZW=(2au)^{2}>0\) and \(Z,W<0\), their units must both be \(-1\); hence there exist integers \(m,n\) such that
\[
Z=-(m^{2}),\qquad W=-(n^{2}).
\]
From \(W=1-9\Delta^{2}=-(n^{2})\) we get
\[
(3\Delta)^{2}-n^{2}=1\quad\Longleftrightarrow\quad (3\Delta-n)(3\Delta+n)=1.
\]
The only factorizations of \(1\) in \(\mathbb{Z}\) are \(1\cdot 1\) and \((-1)\cdot(-1)\).
Both cases give \(3\Delta=\pm 1\), which is impossible for integer \(\Delta\).
(Equivalently, the only integer solutions of \(x^{2}-y^{2}=1\) are \(x=\pm 1,\,y=0\).)

Therefore \eqref{eq:X1-final} has no solutions, and the case \(X\) odd is impossible in Sub-branch~II.2.
Combining with the even cases for \(X\) treated above, Sub-branch~II.2 is closed.

Thus Branch \(3\nmid au\) is impossible.

\section{Completion of the Proof}

We have shown that equation \eqref{eq:star} has no integer solutions \(X\) either when \(3\mid au\) or when \(3\nmid au\).
By Theorem \ref{thm:equiv}, any \(4{+}4\) factorization yields a solution of \eqref{eq:star}; since \eqref{eq:star} has no integer solutions, a \(4{+}4\) factorization is impossible.

\begin{theorem}[Main result]
For any coprime integers \(a\neq u>0\), the polynomial \(P_{a,u}(t)\) does not factor in \(\mathbb{Z}[t]\) as a product of two monic polynomials of degree \(4\).
\end{theorem}

% ===================== Excluding the 2 + 6 Factorization =====================

\section{Excluding a \(2{+}6\) Factorization: a Direct Criterion and a Discriminant Argument}\label{sec:exclude-26}

Recall the notation
\[
P_{a,u}(t)=t^{8}+A t^{6}+B t^{4}+C t^{2}+D,\qquad
A=6\Delta,\quad \Delta:=u^{2}-a^{2}\ne 0,
\]
\[
B=\Delta^{2}-2A_{0},\quad C=-A_{0}A,\quad D=A_{0}^{2},\qquad A_{0}:=a^{2}u^{2}.
\]
Thus \(P_{a,u}\) is even, monic, primitive in \(\mathbb{Z}[t]\) and admits the representation
\begin{equation}\label{eq:P-as-Q}
P_{a,u}(t)=Q(t^{2}),\qquad Q(x):=x^{4}+Ax^{3}+Bx^{2}+Cx+D\in\mathbb{Z}[x].
\end{equation}

We show that a factorization of type \(2{+}6\) is impossible.

\subsection*{Step 0: Structural split of the class \(2{+}6\)}

Suppose
\[
P_{a,u}(t)=Q_{2}(t)\cdot H_{6}(t),\qquad \deg Q_{2}=2,\ \deg H_{6}=6.
\]
By evenness of \(P_{a,u}\) and the involution \(t\mapsto -t\) (Lemma~\ref{lem:types}), we have:
\begin{itemize}[label={},leftmargin=0em]
\item If \(Q_{2}\) is \emph{not} even, then necessarily \(Q_{2}(-t)\mid H_{6}(t)\).
Grouping the conjugate factors, we obtain an even quartic and another quartic:
\[
P_{a,u}(t)
= \underbrace{Q_{2}(t)\,Q_{2}(-t)}_{\text{degree }4,\ \text{even}}
\;\cdot\;
\underbrace{\dfrac{H_{6}(t)}{Q_{2}(-t)}}_{\text{degree }4},
\]
i.e. the factorization regroups to the case \(4{+}4\), which has already been excluded.
\item Hence the only residue to analyze is the \emph{even} quadratic
\[
Q_{2}(t)=t^{2}+q,\qquad q\in\mathbb{Z}.
\]
\end{itemize}

We now rule out this last possibility by a direct necessary and sufficient condition plus a discriminant computation.

\subsection*{Step 1: Criterion for an even quadratic divisor}

\begin{lemma}[Even quadratic divisor criterion]\label{lem:even-quadratic-criterion}
For \(q\in\mathbb{Z}\) we have
\[
\boxed{\ (t^{2}+q)\ \mid\ P_{a,u}(t)\ \Longleftrightarrow\ Q(-q)=0\ },
\]
where \(Q\) is as in \eqref{eq:P-as-Q}. In other words,
\[
(t^{2}+q)\mid P_{a,u}(t)\ \Longleftrightarrow\ q^{4}-Aq^{3}+Bq^{2}-Cq+D=0.
\]
\end{lemma}

\begin{proof}
Divide \(Q(x)\) by \(x+q\) in \(\mathbb{Z}[x]\):
\(
Q(x)=(x+q)R(x)+S
\)
with \(R\in\mathbb{Z}[x]\) and a constant remainder \(S=Q(-q)\).
Substituting \(x=t^{2}\) and using \eqref{eq:P-as-Q} gives
\[
P_{a,u}(t)=Q(t^{2})=(t^{2}+q)\,R(t^{2})+S.
\]
Thus \((t^{2}+q)\mid P_{a,u}\) if and only if \(S=0\), i.e. \(Q(-q)=0\).
\end{proof}

\begin{remark}
The case \(q=0\) is automatically impossible: if \(t^{2}\mid P_{a,u}(t)\), then the constant term must vanish, but \(D=A_{0}^{2}=a^{4}u^{4}>0\).
\end{remark}

\subsection*{Step 2: A discriminant obstruction}

We rewrite the equality \(Q(-q)=0\) from Lemma~\ref{lem:even-quadratic-criterion} as a quadratic equation in the unknown \(A_{0}=a^{2}u^{2}\) while \(\Delta\) and \(q\) are regarded as fixed integers. Using
\(
A=6\Delta,\ B=\Delta^{2}-2A_{0},\ C=-A_{0}A=-6\Delta A_{0},\ D=A_{0}^{2},
\)
we compute
\[
\begin{aligned}
Q(-q)&=q^{4}-Aq^{3}+Bq^{2}-Cq+D\\
&=q^{4}-6\Delta q^{3}+(\Delta^{2}-2A_{0})q^{2}+6\Delta A_{0}q+A_{0}^{2}\\
&=\underbrace{A_{0}^{2}}_{\text{quadratic in }A_{0}}
+\underbrace{(6\Delta q-2q^{2})}_{=:b}\,A_{0}
+\underbrace{(\Delta^{2}q^{2}-6\Delta q^{3}+q^{4})}_{=:c}.
\end{aligned}
\]
Thus \(Q(-q)=0\) is the quadratic equation in \(A_{0}\):
\[
A_{0}^{2}+b\,A_{0}+c=0,\qquad b=6\Delta q-2q^{2},\quad c=\Delta^{2}q^{2}-6\Delta q^{3}+q^{4}.
\]
Its discriminant with respect to \(A_{0}\) equals
\[
\begin{aligned}
\Disc_{A_{0}}
&=b^{\,2}-4c
=(6\Delta q-2q^{2})^{2}-4(\Delta^{2}q^{2}-6\Delta q^{3}+q^{4})\\
&=\bigl(36\Delta^{2}q^{2}-24\Delta q^{3}+4q^{4}\bigr)
-\bigl(4\Delta^{2}q^{2}-24\Delta q^{3}+4q^{4}\bigr)\\
&=\boxed{\,32\,\Delta^{2}\,q^{2}\,}.
\end{aligned}
\]

\begin{proposition}[Irrationality of the would-be roots]\label{prop:disc}
If \(\Delta\neq 0\) and \(q\neq 0\), then \(\Disc_{A_{0}}=32\,\Delta^{2}\,q^{2}\) is \emph{not} a perfect square in \(\mathbb{Z}\).
\end{proposition}

\begin{proof}
We have \(\nu_{2}(\Disc_{A_{0}})=\nu_{2}(32)+2\nu_{2}(\Delta q)=5+2\nu_{2}(\Delta q)\), which is odd for all \(\Delta q\neq 0\).
A perfect square in \(\mathbb{Z}\) must have even \(2\)-adic valuation. Hence \(\Disc_{A_{0}}\) is not a square in \(\mathbb{Z}\).
\end{proof}

\begin{remark}
This “odd $2$-adic valuation of the discriminant forces non-squareness” obstruction is a standard device in elementary Diophantine arguments; compare also the problem-oriented expositions in \cite{MurtyEsmonde}.
\end{remark}

\begin{corollary}[No integer solution for \(A_{0}\)]\label{cor:no-int-A0}
For \(\Delta\neq 0\) and \(q\neq 0\) the quadratic equation \(A_{0}^{2}+bA_{0}+c=0\) has no solutions \(A_{0}\in\mathbb{Z}\).
\end{corollary}

\begin{proof}
The roots are \(\dfrac{-b\pm\sqrt{\Disc_{A_{0}}}}{2}\); by Proposition~\ref{prop:disc} the discriminant is not an integer square, hence the roots are irrational.
\end{proof}

\subsection*{Step 3: Conclusion for \(2{+}6\)}

\begin{theorem}[No \(2{+}6\) factorization]\label{thm:no-2-6}
Let \(a,u\in\mathbb{Z}_{>0}\) be coprime and \(a\ne u\) (so \(\Delta\ne 0\)).
Then \(P_{a,u}(t)\) does not factor in \(\mathbb{Z}[t]\) as a product of a quadratic and a sextic polynomial.
\end{theorem}

\begin{proof}
As noted above, any \(2{+}6\) with a non-even quadratic regroups to a \(4{+}4\), which is impossible.
Thus it remains to exclude an even quadratic \(t^{2}+q\).
By Lemma~\ref{lem:even-quadratic-criterion}, \((t^{2}+q)\mid P_{a,u}\) iff \(Q(-q)=0\).
If \(q=0\), divisibility by \(t^{2}\) would force \(D=0\), which is false. If \(q\ne 0\), then by
Corollary~\ref{cor:no-int-A0} the equality \(Q(-q)=0\) has no solutions \(A_{0}=a^{2}u^{2}\in\mathbb{Z}\).
Hence there is no \(q\in\mathbb{Z}\) for which \(t^{2}+q\) divides \(P_{a,u}\).
Therefore no \(2{+}6\) factorization exists.
\end{proof}

\begin{remark}[What this uses from previous sections]
The proof is logically independent of the \(4{+}4\) Diophantine analysis, except for the purely structural observation that a non-even quadratic factor forces regrouping into \(4{+}4\)
(via pairing \(Q_{2}(t)\) with its conjugate \(Q_{2}(-t)\)).
The “hard” residue (even quadratic \(t^{2}+q\)) is completely settled by Lemma~\ref{lem:even-quadratic-criterion} and the discriminant computation.
\end{remark}

% ===================== Excluding other factorizations =====================

\section{Excluding other factorizations}

After Theorem~\ref{thm:no-2-6} has ruled out all factorizations of type \(2{+}6\), the remaining degree–8 patterns are excluded by trivial regrouping.

\begin{proposition}\label{prop:others-to-26}
Let $P_{a,u}(t)\in\mathbb{Z}[t]$ be as above. If any of the following factorizations exists,
then $P_{a,u}$ admits a factorization of type \(2{+}6\):
\begin{enumerate}[label=(\alph*),nosep]
\item \(2{+}2{+}4\):\quad $P_{a,u}=Q_1\,Q_2\,H_4$ with $\deg Q_i=2$, $\deg H_4=4$;
\item \(2{+}2{+}2{+}2\):\quad $P_{a,u}=Q_1\,Q_2\,Q_3\,Q_4$ with $\deg Q_i=2$;
\item \(3{+}3{+}2\):\quad $P_{a,u}=F_3\,G_3\,Q_2$ with $\deg F_3=\deg G_3=3$, $\deg Q_2=2$.
\end{enumerate}
\end{proposition}

\begin{proof}
(a) Group as $P_{a,u}=\underbrace{Q_1}_{\deg=2}\cdot\underbrace{(Q_2H_4)}_{\deg=6}$.

(b) Group as $P_{a,u}=\underbrace{Q_1}_{\deg=2}\cdot\underbrace{(Q_2Q_3Q_4)}_{\deg=6}$.

(c) Group as $P_{a,u}=\underbrace{Q_2}_{\deg=2}\cdot\underbrace{(F_3G_3)}_{\deg=6}$.
\end{proof}

\begin{corollary}\label{cor:others-impossible}
None of the patterns \(2{+}2{+}4\), \(2{+}2{+}2{+}2\), or \(3{+}3{+}2\) can occur for \(P_{a,u}(t)\).
\end{corollary}

\begin{proof}
By Proposition~\ref{prop:others-to-26} each would imply a \(2{+}6\) factorization, which is impossible by Theorem~\ref{thm:no-2-6}.
\end{proof}

% ===================== Irreducibility =====================

\section{Irreducibility in Full}\label{sec:irreducibility}

\begin{theorem}[Irreducibility]\label{thm:irreducible}
For any coprime integers $a\ne u>0$, the polynomial $P_{a,u}(t)$ is irreducible in $\mathbb{Z}[t]$.
\end{theorem}

\begin{proof}
All degree-$8$ splittings are excluded as follows.

(i) The case $4{+}4$ is impossible by Theorem~\ref{thm:equiv} and the analysis of equation \eqref{eq:star} (from Lemma~\ref{lem:types} to Corollary~\ref{cor:coprime} and the subsequent $2$-/$3$-adic split).

(ii) The case $2{+}6$ is excluded in Section~\ref{sec:exclude-26}.

(iii) After (ii), any of the remaining patterns $2{+}2{+}4$, $2{+}2{+}2{+}2$, $3{+}3{+}2$ would regroup to $2{+}6$ by Proposition~\ref{prop:others-to-26}, hence are impossible by (ii).

Therefore no nontrivial factorization in $\mathbb{Z}[t]$ exists. Since $P_{a,u}(t)$ is monic and primitive, irreducibility over $\mathbb{Z}$ follows.
\end{proof}

% ===================== Conclusions =====================

\section*{Conclusions}
We have shown that for any coprime integers $a\ne u>0$ the even cuboid polynomial $P_{a,u}(t)$ admits no factorization of type $4{+}4$ in $\mathbb{Z}[t]$. The key step is the reduction of a potential factorization to the Diophantine condition $(X^{2}-8\Delta^{2})(X^{2}-9\Delta^{2})=4a^{2}u^{2}X^{2}$, from which, using $2$- and $3$-adic estimates and the lemma $\gcd(X,\Delta)=1$, the absence of integer solutions follows. 
We then closed the genuine $2{+}6$ case via an exact divisor criterion combined with a discriminant obstruction. 
Finally, after excluding $2{+}6$, any remaining patterns ($2{+}2{+}4$, $2{+}2{+}2{+}2$, $3{+}3{+}2$) regroup trivially to $2{+}6$ and are therefore impossible. 
Altogether, $P_{a,u}(t)$ admits no nontrivial factorization in $\mathbb{Z}[t]$, establishing irreducibility in full~\cite{Sharipov2011Cuboids,Sharipov2011Note,GuyUPINT}.

% ===================== References =====================

% ===================== The end =====================

\end{document}